\documentclass[reqno]{amsart}
\usepackage{bbm}
\usepackage{mathrsfs}
\usepackage{color}
\usepackage{amsmath}
\usepackage{amsfonts}
\usepackage{mathtools}
\usepackage{amssymb}
\usepackage{graphicx}
\usepackage{hyperref}

\hypersetup{
	colorlinks=true,
	linkcolor=blue,
	filecolor=red,      
	urlcolor=cyan,
	citecolor=green,
}

 \newtheorem{theorem}{Theorem}[section]
 \newtheorem{question}[theorem]{Question}
 \newtheorem{remark}[theorem]{Remark}
 
 \numberwithin{equation}{section}

\def\Psh{\mathrm{Psh}}
\def\loc{\mathrm{loc}}

\def\H{\mathcal{H}}
\def\O{\mathcal{O}}

\def\bbC{\mathbb{C}}

\def\e{\mathrm{e}}
\def\d{\mathrm{d}}

\def\wt{\widetilde}

\DeclareMathOperator{\ess}{ess}

\DeclareMathOperator{\Vol}{Vol}
\DeclareMathOperator{\esssup}{\ess\sup}

\DeclareMathOperator{\dist}{dist}

\begin{document}

\title[Demailly's approximation of general weights]
 {Demailly's approximation of general weights}

\author{Shijie Bao}
\address{Shijie Bao: Institute of Mathematics, Academy of Mathematics and Systems Science, Chinese Academy of Sciences, Beijing 100190, China}
\email{bsjie@amss.ac.cn}

\author{Qi'an Guan}
\address{Qi'an Guan: School of
Mathematical Sciences, Peking University, Beijing 100871, China}
\email{guanqian@math.pku.edu.cn}

\subjclass[2020]{32A25, 32U05}

\keywords{Demailly's approximation, plurisubharmonic function, Bergman kernel}

\date{}


\begin{abstract}
In this note, we demonstrate the convergence of the Demailly approximation of a general (weakly) upper semi-continuous weight.
\end{abstract}

\maketitle

\section{Introduction}

Let $\Omega$ be a bounded pseudoconvex domain in $\bbC^n$, and let $V\colon \Omega\to [-\infty, +\infty)$ be a Lebesgue measurable function. The Bergman space associated with the weight $\e^{-2mV}$ on $\Omega$ is a Hilbert space defined by
\[\H_{\Omega}(mV)\coloneqq\big\{f\in \O(\Omega)\mid \|f\|_{mV}<+\infty\big\},\]
where the norm is given by
\[\|f\|_{mV}\coloneqq \left(\int_{\Omega}|f|^2e^{-2mV}\d\lambda\right)^{1/2}.\]
For each $m\ge 1$, the \textbf{Demailly approximation} $V_m$ of $V$ is defined as
\begin{equation}\label{eq-Dem.approx}
    V_m(z)\coloneqq\frac{1}{2m}\log K_{mV}(z)\coloneqq\frac{1}{2m}\log\sup_{\substack{f\in \H_{\Omega}(mV),\\ \|f\|_{mV}=1}}|f(z)|^2, \quad \forall z\in\Omega,
\end{equation}
where $K_{mV}$ denotes the Bergman kernel for the weight $e^{-2mV}$ on $\Omega$.

When $V\in\Psh(\Omega)$, i.e., $V$ is plurisubharmonic (psh) on $\Omega$, Demailly's approximation theorem (see \cite{Dem92, AMAG}) asserts that $V_m\to V$ pointwise and in $L^1_{\loc}$ on $\Omega$ as $m\to+\infty$:

\begin{theorem}[Demailly's approximation theorem \cite{Dem92}]\label{thm-Demailly.approx}
    Let $V\in\Psh(\Omega)$. Then there are constants $C_1, C_2>0$ independent of $m$ such that
    \[V(z)-\frac{C_1}{m}\le V_m(z)\le\sup_{|\zeta-z|<r}V(\zeta)+\frac{1}{m}\log\frac{C_2}{r^n}\]
    for every $z\in\Omega$ and $r<\dist(z, \partial\Omega)$. In particular, $V_m$ converges to $V$ pointwise and in $L_{\loc}^1$ topology on $\Omega$ when $m\to +\infty$.
\end{theorem}

The left-hand inequality in Theorem \ref{thm-Demailly.approx} follows from the Ohsawa-Takegoshi $L^2$ extension theorem (\cite{OT87}). Demailly's approximation theorem provides a systematic framework for approximating general psh weights by those with analytic singularities, a technique with profound implications in algebraic geometry (e.g. \cite{Dem93, DK01, Cao14}). For instance, Theorem \ref{thm-Demailly.approx} immediately yields Siu's theorem (\cite{Siu74}) on the analyticity of Lelong number upper-level sets, which states that for every $V\in\Psh(\Omega)$ and $c>0$, the set
\[E_c(V)\coloneqq\big\{z\in \Omega\mid \nu(V,z)\ge c\big\}\]
is an analytic subset of $\Omega$. See \cite[Chapter 14]{AMAG} for detailed proofs and further applications of Demailly's approximation theorem.

\subsection{Demailly's approximation of upper semi-continuous weights}

A natural question arises: \emph{Does Demailly's approximation converges to a specific function in more general settings}? This problem and its solution might be probably known to experts, but in the literature accessible to us, we have not found any direct statements regarding the results of this issue. For smooth weights, Matsumura proposed the following more precise question during an intensive course at Peking University:

\begin{question}[\cite{Matsu24}]\label{conj-Matsu}
    Let $V$ be a smooth function on $\Omega$, and define
    \[\wt{V}\coloneqq\sup\big\{\psi\in\Psh(\Omega)\mid \psi\le V \text{ on } \Omega\big\}.\]
    Does $V_m$ converge to $\wt{V}$ in some sense?
\end{question}

In this paper, we affirmatively answer Question \ref{conj-Matsu}, and extend the results to non-smooth weights. Our main theorem is stated as follows:

\begin{theorem}[Main Theorem]\label{thm-main.thm}
    Let $\Omega$ be a bounded pseudoconvex domain in $\bbC^n$, and $V\colon \Omega \to [-\infty, +\infty)$ be a Lebesgue measurable function on $\Omega$. For every $m>0$, let $V_m$ be the Demailly's approximation of $V$ defined in (\ref{eq-Dem.approx}). Suppose that $V$ is weakly upper semi-continuous on $\Omega$. Let
    \[\wt{V}\coloneqq\sup\big\{\psi\in\Psh(\Omega)\mid \psi\le V \text{ on } \Omega\big\}.\]
    Then $V_m$ converges to $\wt{V}$ pointwise as $m\to +\infty$. 
\end{theorem}

We call that $V$ is \textbf{weakly upper semi-continuous} on $\Omega$, if the \textbf{(weakly) upper semi-continuous regularization} of $V$ defined by:
    \[V^{\star}(a)\coloneqq \lim_{r\to 0^+}\Big(\mathop{\esssup}\limits_{z\in B(a,r)}V(z)\Big), \quad \forall a\in \Omega,\]
equals to $V$ itself, where `$\esssup$' denotes the essential supremum.

\begin{remark}
    If $V$ is continuous on $\Omega$, then $V^{\star}\equiv V$, and Theorem \ref{thm-main.thm} implies $V_m\to \wt{V}$ pointwise, thereby answering Question \ref{conj-Matsu}.
\end{remark}

\vspace{.1in} {\em Acknowledgements}.
The authors would like to thank Professor Shin-ichi Matsumura for his talk at Peking University, which brought this question to our attention. The first author would also like to express gratitude to Professor Shin-ichi Matsumura, and Ms. Bowen Chen for their valuable discussions, and to Professor Wang Xu for pointing out that the example in the first version of this paper is not appropriate. The second author was supported by National Key R\&D Program of China 2021YFA1003100, NSFC-12425101 and NSFC-11825101.

\section{Proof of the Main Theorem}
Now we prove Theorem \ref{thm-main.thm}.
\begin{proof}[Proof of Theorem \ref{thm-main.thm}]
    Note that $V$ is locally upper bounded on $\Omega$ due to its weakly upper semi-continuity. Then we can let $(\wt{V})^{\star}$ be the (weakly) upper semi-continuity regularization of $\wt{V}$.  Observe that $(\wt{V})^{\star}\le V^{\star}=V$, since $V=V^{\star}$ and $\wt{V}\le V$. Meanwhile, we  have $(\wt{V})^{\star}\in\Psh(\Omega)$ (see \cite[Proposition 4.24]{CADG}, and see also \cite{Dem85}). It follows that $\wt{V}=(\wt{V})^{\star}\in\Psh(\Omega)$ by the definition of $\wt{V}$.

    Define
    \[\Phi(z)\coloneqq\Big(\limsup_{m\to+\infty}V_m(z)\Big)^{\star}, \quad \forall z\in\Omega.\]
    Since every $V_m$ is a psh function on $\Omega$, the function $\Phi\in \Psh(\Omega)$. We prove $\Phi\le \wt{V}$ on $\Omega$. For every $f\in\H_{\Omega}(mV)$, the mean value inequality applied to $|f|^2$ yields
    \begin{equation*}
            |f(z)|^2\le\frac{1}{\Vol(B(z,r))}\int_{B(z,r)}|f|^2\le\frac{1}{\Vol(B(z,r))}\|f\|^2_{mV}\cdot\mathop{\esssup}\limits_{B(z,r)}e^{2mV},
    \end{equation*}
    which implies
    \[K_{mV}(z)\le\frac{1}{\Vol(B(z,r))}\cdot\mathop{\esssup}\limits_{B(z,r)}e^{2mV},\]
    and thus
    \[V_m(z)\le\mathop{\esssup}\limits_{B(z,r)}V+\frac{1}{m}\log\frac{C_2}{r^n},\]
    where $C_2>0$ is a constant depending only on $n$, and $r<\dist(z, \partial\Omega)$. Taking $m\to+\infty$ and $r\to 0^+$, we obtain $\Phi\le V^{\star}=V$. By the definition of $\wt{V}$, it follows that $\Phi\le\wt{V}$, and thus
    \begin{equation}\label{eq-limsup}
    \limsup_{m\to+\infty}V_m(z)\le \Phi(z)\le \wt{V}(z), \quad \forall z\in\Omega.
    \end{equation}

    On the other hand, since $\wt{V}\le V$ on $\Omega$ by the definition, every holomorphic function $F\in\H_{\Omega}(m\wt{V})$ belongs to $\H_{\Omega}(mV)$, with $\|F\|_{m\wt{V}}\ge\|F\|_{mV}$. It follows that
    $K_{m\wt{V}}(z)\le K_{mV}(z)$. Let 
    \[\wt{V}_m(z)\coloneqq\frac{1}{2m}\log K_{m\wt{V}}(z), \quad \forall z\in\Omega.\]
    Then $V_m\ge\wt{V}_m$  on $\Omega$ for every $m>0$. Using Demailly's approximation theorem (or the Ohsawa-Takegoshi $L^2$ extension theorem \cite{OT87}), we obtain
    \begin{equation*}
        \wt{V}_m(z)\ge\wt{V}(z)-\frac{C_1}{m},  \quad \forall z\in\Omega,
    \end{equation*}
    where $C_1$ is a constant depending only on $\Omega$ and $n$. Since $V_m\ge\wt{V}_m$, it follows that
    \[\liminf_{m\to+\infty}V_m(z)\ge \liminf_{m\to+\infty} \wt{V}_m(z)\ge \wt{V}(z), \quad \forall z\in\Omega.\]
    Hence, combining with (\ref{eq-limsup}), we get $V_m(z)\to\wt{V}(z)$ pointwise.
\end{proof}

\end{document}